\documentclass[preprint,12pt]{amsart}

\usepackage{amssymb}

\newtheorem{theorem}{Theorem}
\newtheorem{lemma}[theorem]{Lemma}
\newtheorem{remark}[theorem]{Remark}

\def\cbd{\hfill$\Box$}

\begin{document}

\title{Li-Yorke sensitive and weak mixing dynamical systems}

\author{Michaela Ml\'{\i}chov\'a}

\address{Mathematical Institute, Silesian University, 746 01 Opava, Czech Republic}
\date{\today}

\begin{abstract}

Akin and Kolyada in 2003 [E. Akin, S. Kolyada,  Li-Yorke sensitivity, Nonlinearity 16  (2003) 1421 -- 1433] introduced the notion of Li-Yorke sensitivity. They proved that every weak mixing system $(X, T)$, where $X$ is a compact metric space and $T$ a continuous map of $X$ is Li-Yorke sensitive.  An example of Li-Yorke sensitive system without weak mixing factors was given  in  [M. \v Ciklov\'a, Li-Yorke sensitive minimal maps, Nonlinearity 19 (2006) 517 -- 529]
(see also [M. \v Ciklov\'a-Ml\'{\i}chov\'a, Li-Yorke sensitive minimal maps II, Nonlinearity 22 (2009) 1569 -- 1573]). In their paper, Akin and Kolyada  conjectured that every minimal system with a weak mixing factor, is Li-Yorke sensitive. We provide arguments supporting this conjecture though the proof seems to be difficult.

\end{abstract}

\maketitle

\section{Introduction}

A {\it topological dynamical system} $(X, T)$ is a compact metric space $(X, \rho)$ endowed with a continuous surjective map $T: X \rightarrow X$. Denote by $T^n$ the $n$th iterate of $T$, $n\ge 0$. Points $x,y\in X$ are {\it proximal}, or  {\it $\delta$-asymptotic} (with $\delta \ge 0$), or {\it distal} if $\liminf _{n\to\infty}\rho (T^n(x),T^n(y))=0$, $\limsup _{n\to\infty} \rho (T^n(x), T^n(y))\le\delta$, or $\liminf _{n\to\infty} \rho (T^n(x), T^n(y))>0$, respectively; instead of $0$-asymptotic we say {\it asymptotic}. A map $T:X \rightarrow X$ is {\it Li-Yorke sensitive}, briefly $LYS$ or $LYS_\varepsilon$, if there is an $\varepsilon>0$ with the property  that every $x \in X$ is a limit of points $y \in X$ such that the pair $(x, y)$ is {\it proximal} but not {\it $\varepsilon$-asymptotic}, i.e.,  if
\begin{equation}
\label{e:lys}
\liminf_{n \to\infty} \rho(T^n(x), T^n(y))=0 ,\quad {\rm and}\quad \limsup_{n\to \infty} \rho(T^n(x), T^n(y))>\varepsilon.
\end{equation}
Every pair $(x, y) \in X \times X$  satisfying (\ref{e:lys})  is an {\it $\varepsilon$-Li-Yorke pair}. A  set $S\subseteq X$ such that any points  $x\ne y$ in $S$ satisfy (\ref{e:lys}) is an ($\varepsilon$-){\it scrambled set}. A map $T$ is {\it Li-Yorke chaotic}, briefly  $LYC$ or $LYC_\varepsilon$ if it has an uncountable $\varepsilon$-scrambled set, for some $\varepsilon >0$.

A system $(X, T)$ is {\it transitive} if for every pair of open, nonempty subsets $U, V \subset X$ there is a positive integer $n$ such that $U \cap T^{-n} (V) \neq \emptyset$; it is {\it weakly mixing} if the product system $(X \times X, T \times T)$ is transitive; it is {\it minimal} if every point $x \in X$ has a dense orbit $\{T^n(x)\}_{n=0}^\infty$. Finally, a system $(Y, S)$ is a {\it factor} of $(X, T)$ if there is a surjective continuous map $\pi: X \rightarrow Y$ such that $\pi \circ T=S\circ \pi$.  In this case we say that $(X, T)$ is an {\it extension} of $(Y, S)$. A {\it skew-product system} is a system $(X\times Y, F)$ where $X, Y$ are compact metric spaces, and $F$ a continuous map such that $F(x,y)=(f(x), g_x(y))$, for every $x\in X, y\in Y$. Other notions will be defined later, or can be found in \cite{A} or in related papers listed in references.
 
 The notion of $LYS$ was introduced and studied by Akin and Kolyada \cite{AK}. It turns out that such systems are related to weak mixing systems. For example, every nontrivial weak mixing system is $LYS$. Therefore, Akin and Kolyada stated in \cite{AK} five conjectures  concerning $LYS$ systems. Three of them were disproved in \cite{M1} and \cite{M2}. In particular, it was proved that a minimal $LYS$ system need not have a nontrivial weak mixing factor, and that a minimal system with a nontrivial $LYS$ factor need not be $LYS$. The remaining two open problems are the following: \\

\noindent {\bf P1.} Is every minimal system with a nontrivial weak mixing factor $LYS$?   \\
\noindent {\bf P2.} Does $LYS$ imply $LYC$?\\

 \noindent Both problems seem to be difficult but, in contrast to the preceding ones, it seems that the answer is in both cases positive. In this paper we give  partial solutions. Recall that the only known result related to (P1) is the following
                 
\begin{theorem}
\label{th1}(See \cite{AK}, \cite{B}.) If $(X,T)$ is minimal weak mixing and $(Y,S)$ minimal and distal then $(X\times Y, T\times S)$ is minimal and LYS.
\end{theorem}

\noindent  We generalize it to some skew-product extensions of the original system  in Theorems \ref{main1}, \ref{main2} which represent the main results of this paper. Then, in Theorem \ref{main3} we show that the restriction to skew-product extensions is not too limiting. Finally, our last result, Theorem \ref{main4},  essentially diminishes the possible class of systems which may not satisfy (P2). For convenience, we recall several known results which will be of use in the next sections. 

\begin{lemma}
\label{int} (See \cite{KST}.)  If $(X,T)$ is a minimal system then, for every open set $G\subseteq X$ there is an open set $H\subseteq X$ such that $H\subseteq T(G)\subseteq \overline H$.
\end{lemma}

\begin{lemma}
\label{lyp} (See \cite{AK}.) If $(X,T)$ is weak mixing then, for every $x\in X$, the set  of points $y\in X$ which are proximal but not $\varepsilon$-asymptotic to $x$, is a dense $G_\delta$ subset of $X$.
\end{lemma}

\begin{lemma}
\label{uncount}  (See \cite{HY1}.) Let $X$ be a complete separable metric space without isolated points. If $R\subseteq X\times X$ is a symmetric relation with the property that  for each $x\in X$, $R(x)=\{ y\in X; (x,y)\in R\}$ contains a dense $G_\delta$ subset, then there is a dense uncountable set $D\subseteq X$ such that $D\times D\setminus\Delta\subset R$, where $\Delta$ is the set of pairs $(x,x)$, $x\in X$.
\end{lemma}

The following is a topological version of the Fubini Theorem.

\begin{lemma}
\label{FT} (See \cite{A} or \cite{HY1}.) Let $R$ 
be a relation on a complete separable metric space $X$ which contains a dense $G_\delta$ subset of $X\times X$.
 Then there is a dense $G_\delta$ set $A\subseteq X$ 
 such that for each $x\in A$, there exists a dense 
 $G_\delta$ set $X_x\subseteq X$  with $\{ (x,y); x\in A, y\in X_x\}\subseteq R$.
 \end{lemma} 

\begin{lemma}
\label{asym} (See \cite{HY1} and \cite{AK}.) If $(X,T)$ is $LYS$ then, for some $\delta >0$, the set of $\delta$-asymptotic pairs is a first category subset of $X\times X$.
\end{lemma}

\section{Minimal finite-type skew-product extensions of weak mixing systems}

\begin{lemma}
\label{min} Let $(X,T)$ be minimal, $A=\{ a_1,\cdots ,a_m\}$ a finite set with discrete topology, $Y=X\times A$, $S$ a skew-product map $Y\to Y$ so that $S(u,v)=(T(u), G_u(v))$, and $M\subseteq Y$ a minimal set. Then
\begin{enumerate}
    \item for every  $u\in X$,  the map $G_u$ restricted to the set  $M_u:=M\cap (\{ u\}\times A)$, is injective;
    \item there is a nonempty set $B\subseteq A$ such that $(M, S|_M)$ is conjugate to $(X\times B, S|_{X\times B})$.
\end{enumerate}
\end{lemma}    
   
\begin{proof} (i) Assume $G_u(v_1)
=G_u(v_2)=v$, for some $u$ in $X$ and 
$v_1\ne v_2$ in $A$. Since $A$ is discrete,
by the continuity of $S$ there is a neighborhood 
$U$ of $u$ such that, for every $w\in U$, 
$G_w(v_1)=G_w(v_2)=v$. Then  $U_1=U\times
\{v_1\}$ and $U_2=U\times\{v_2\}$ would be disjoint 
nonempty open sets with $S(U_1)=S(U_2)$. 
But this is impossible, by Lemma \ref{int}. 

(ii) Let $\mathcal H=\{ h_1,\cdots , h_l\}$ be the collection of   maps $G_u|_{M_u}$, $u\in X$. By the continuity, there is a decomposition of $X$ into clopen sets $X_1,\cdots , X_l$ such that, for every $u\in X_j$, $G_u|_{M_u}=h_j$. Let $c_i$ be the number of points in the domain of $h_i$ and let, say, $c_1\ge c_i$, for every $i$. Since $T$ is transitive, (i) implies $c_i=c_1$, for every $i$. It follows that $M=X_1\times A_1\cup X_2\times A_2\cup\cdots\cup X_l\times A_l$, where  $\#A_j=c_1$, for every $j$. Hence $M$ is conjugate to $X\times A_1$.
\cbd
\end{proof}

\medskip

\begin{lemma}
\label{trans} If $(X,T)$ is minimal weak mixing, then
for every $x\in X$ the set ${\rm Tran}(x)\subset X$ of points $y$ such that $(x,y)$ is a transitive point with respect to $T\times T$, is a dense $G_\delta$ set.
\end{lemma}

\begin{proof} 
Let $x_0\in X$ be given.  The set ${\rm Tran}(x_0)$ of points $y\in X$ such that $(x_0,y)$ is a transitive point of $T\times T$, is a $G_\delta$ set since it is the intersection of two $G_\delta$ sets, the set of transitive points $(x,y)\in X\times X$, and $\{x_0\}\times X$. So it suffices to show that ${\rm Tran}(x_0)$ is dense in $X$. Let $\{ G_n\} _{n\ge 1}$ be a  base of open sets for $X\times X$ of the form $G_n=I_n\times J_n$, where $I_n,J_n$ are open sets. Let $U_0\subset X$ be nonempty open. By induction, 
there are nonempty open 
sets $U_0\supset U_1\supset U_2\supset \cdots$, and a sequence $n_1<n_2<\cdots$ of positive 
integers such that
\begin{equation}
\label{eq99}
{\overline  U_{j}}\subset U_{j-1}, \ T^{n_j}(x_0)\in I_j \ {\rm and} \  T^{n_j}(U_j) \subset J_j, \ j\in\mathbb N.
\end{equation}

Indeed, since $T$ is minimal, there is a $k_1>0$ such that, for every $j$, there is an $s$, $0\le s <k_1$, with $T^{j+s}(x_0)\in I_1$. Since $T$ is weak mixing, the set $N(U_0,J_1)$ of times $i$ such that $T^i(U_0)\cap J_1\ne\emptyset$, contains arbitrarily long blocks of successive integers. It follows that there is an $n_1$ such that $T^{n_1}(x_0)\in I_1$ and $T^{n_1}(U_0)\cap J_1\ne\emptyset$. Since $T$ is minimal and $U_0$ open, $T^{n_1}(U_0)$ has nonempty interior (see Lemma \ref{int}) and hence $T^{n_1}(U_0)\cap J_1$ contains a nonempty open set $H$. It suffices to take for $U_1$ a nonempty open set such that ${\overline U_1}\subset T^{-n_1}(H)$. Thus, we have $n_1$ and $U_1$ satisfying (\ref{eq99}) (for $j=1$). Next we apply the above process with $U_0$ replaced by $U_1$, $G_1$ by $G_2$, obtaining $U_2\subset U_1$ and $n_2>n_1$, etc. This proves (\ref{eq99}). 

To finish the argument put $Y=\bigcap _{j\ge 1}U_j=\bigcap _{j\ge 1}{\overline U_j}$. Then $Y\ne\emptyset$ is a $G_\delta$ set and, by (\ref{eq99}), for every $y\in Y$, $(x_0,y)$ is a transitive point.
\cbd
\end{proof}

\medskip

For $\xi>0$ let $\Delta_\xi\subseteq X\times X$ be the set of pairs $(x,y)$ such that $\rho(x,y)<\xi$. 

\begin{theorem}
\label{lys} 
Let $(X,T)$ be a minimal weak mixing topological dynamical system. Let $A$ be a finite space with discrete topology, $Y=X\times A$ with the max-metric, and  $(Y,S)$ a skew-product extension of $(X,T)$ such that $S(t,a)=(T(t), G_t(a))$, where every fibre map $G_t$ is a bijection of $A$. Then $(Y,S)$ is $LYS_\varepsilon$ for any $0<\varepsilon < {\rm diam} (X)$.  
\end{theorem}

\begin{proof}
The following terminology and notation will be useful. For every $z=(x,y)\in X\times X$ and $i\in\mathbb N$, denote by $(T \times T)^i(z)=(x_i, y_i)$ the $i$th iterate of $z$, with $x_0:=x$, $y_0:=y$. Let $g_0=h_0=Id$, the identity and,  for $i>0$ let $g_i=G_{x_{i-1}}\circ G_{x_{i-2}}\circ\cdots\circ G_{x_0}$, similarly let $h_i$ be the composition of $i-1$ corresponding maps $G_{y_j}$, and let $c_i:=(g_i,h_i)$. For $\eta>0$ let $N= N(x,y,\eta)=\{i\in\mathbb N; (x_i,y_i)\in\Delta_\eta\}$.  The sequence $\{c_i\}_{i\in N}$ is the $\eta$-{\it characteristic sequence} of $(x,y)$. Let $j_0<j_1<\cdots$ be the numbers in $N$. A finite string $c_{j_0}, c_{j_1}, \cdots, c_{j_{k-1}}$  is an $\eta$-{\it saturated chain for} $(x,y)$ {\it of length} $k$ if the string contains all members of the $\eta$-characteristic sequence of $(x,y)$; we denote it as $M(x,y,\eta)$, and we let $C(x,y, \eta)$ denote the set of elements in $M(x,y, \eta)$. Notice that we do not determine uniquely the length of a saturated string: if $M(x,y,\eta)=\{c_{j_0}, \cdots, c_{j_{k-1}}\}$ then $\{c_{j_0}, \cdots, c_{j_k},c_{j_{k}}\}$ is also saturated string. When dealing  with an another pair, $(x^\prime,y^\prime)$, we use primes to distinguish the related symbols like $x^\prime_i, y^\prime_i$, $j^\prime_i, c^\prime_{j^\prime_i}$, $k^\prime$, etc. By the continuity, for every saturated chain $M(x,y,\eta)$ of length $k$ there is an open neighborhood $U(x,y,\eta)$ of $(x,y)$ such that, for  any pair $(x^\prime,y^\prime)\in U(x,y,\eta)$,  $\{(x^\prime_i,y^\prime_i)\}_{0\le i\le j_{k-1}}$ traces $\{(x_i,y_i)\}_{0\le i\le j_{k-1}}$ so that $\rho((x_i,y_i),(x^\prime_i,y^\prime_i))<\eta$ and $(x^\prime_i,y^\prime _i)\in\Delta_\eta$ for $i\in\{j_0, j_1, j_2,\cdots j_{k-1}\}$.  In particular,  $M(x,y,\eta)=M(x^\prime,y^\prime,\eta)$.

Saturated strings $M(x,y,\eta), M(x^\prime,y^\prime, \eta^\prime)$ with $\eta^\prime\le\eta$ of two {\it transitive} pairs $(x,y)$ and $(x^\prime,y^\prime)$ of length $k$ and $k^\prime$, respectively, can be joined in a single  chain  of $(x,y)$ in the following sense.  Since $(x,y)$ is  transitive, there is an $n\ge j_{k-1}$ such that $(x_n,y_n)\in U(x^\prime, y^\prime, \eta^\prime)$. It follows that  $n=j_s$ for some $s\ge k-1$, and the trajectory  $\{ (x_{n+i}, y_{n+i})\}$ traces the trajectory $\{ (x^\prime_{i}, y^\prime_{i})\}$ for $i\in [0,j^\prime_{k^\prime-1}]$, remaining within distance $\eta^\prime\le\eta$ for $i=j^\prime_l,0\le l\le k^\prime-1$.  We denote the resulting  string as $M(x,y,\eta)*M(x^\prime, y^\prime,\eta^\prime)$. Its length is $s+k^\prime $. Thus, we have the following\\

{\bf Claim 1.} {\it Let $\eta \ge \eta^\prime>0$, and $M(x,y,\eta),M(x^\prime, y^\prime, \eta^\prime)$ be saturated strings of transitive pairs, of length $k$ and $k^\prime$, respectively. Then

(i) The string $M(x,y,\eta)*M(x^\prime, y^\prime,\eta^\prime)$ need not be saturated for $\eta$ or $\eta^\prime$, but it is obtained  from $M(x,y,\eta)$ of sufficiently hight length by omitting some elements;

(ii) $C(x,y,\eta)\supseteq C(x^\prime, y^\prime,\eta^\prime)\circ c_{j_s}$,

\noindent where $j_s$ is specified above ($c_{j_s}$ is the element \lq\lq connecting\rq\rq both saturated strings), and $\{f,g\}\circ h$ means $\{f\circ h, g\circ h\}$.}\\

{\bf Claim 2.} {\it Let $(x,y), (x^\prime,y^\prime)$  be transitive pairs.   Then 

(i) $\# C(x,y,\eta)=\# C(x^\prime, y^\prime,\eta)$ for any $\eta >0$;

(ii) there is a $\xi>0$ such that if $(x,y), (x^\prime, y^\prime)\in\Delta_\xi$ and $0<\eta, \eta^\prime <\xi$ then $\# C(x,y,\eta)=\# C(x^\prime, y^\prime,\eta^\prime)$.} \\

{\it Proof of Claim 2.}  (i)  Assume  $m:=\# C(x,y,\eta)<m^\prime: =\# C(x^\prime, y^\prime,\eta)$. Then $\# C(x^\prime, y^\prime, \eta)= m^\prime$ and, by Lemma B(ii), $C(x,y,\eta)$ contains $m^\prime$ distinct elements, a contradiction.

(ii) We may assume $\eta^\prime<\eta$. Then obviously $\#C(x,y,\eta^\prime)\le\#C(x,y,\eta)$. Since $G$ is finite, there is a $\xi >0$, and $m_0\ge 1$ such that $\# C(x,y,\eta)=m_0$ whenever $\eta <\xi$. To finish apply (i).
$\hfill\Box$\\

{\bf Claim 3.} {\it Let $\xi$ be as in Claim 2,  and  $(x,y)\in\Delta_{\xi}$ be a transitive pair, and $\eta^\prime <\eta:= \xi$.   If $c_{j_s}\in C(x,y,\eta)$ is the element connecting the strings  $M(x,y,\eta)$ and $M(x,y,\eta^\prime)$, then $C(x,y,\eta^\prime) \circ c_{j_s}$ contains the identity map $(Id, Id)$.}\\

{\it Proof of Claim 3.} Let $x,y$ be as in the hypothesis. By Claim 1, $M(x,y,\eta)*M(x,y,\eta^\prime)$ contains the string $M(x,y,\eta^\prime)\circ c_{j_s}$.  But $M(x,y,\eta^\prime)\circ c_{j_s}$  must contain the identity. To see this note that, by definition,  the first member $c_0$ of the $\eta$-characteristic sequence is the identity. Since $\rho (x,y)< \eta$ $(=\xi$, $c_0$ is also the first member of $M(x,y,\eta)$, i.e., $j_0=0$. By Claim 2 (ii), $C(x, y, \eta)$ has the same cardinality as $C(x,y,\eta^\prime)$ hence as $C(x,y,\eta^\prime)\circ c_{j_s}$, since $c_{j_s}$ is a bijection. Thus $C(x,y,\eta^\prime) \circ c_{j_s}$ contains the identity map $(Id, Id)$.
$\hfill\Box$\\

Finally, let $x \in X$ and $U$ be an arbitrary neighborhood of $x$. Assume that $\xi>0$ is as in Claim 2. By Lemma \ref{trans} there is a point $y \in U$ such that $(x, y) \in \Delta_\xi$ is transitive with respect to $T \times T$. Since $Y$ is equipped with the $\max$-metric it suffices to prove that if $a\in A$ then, for every $0<\eta^\prime <\xi$,  $w=((x,a)(y,a))$ is an $\eta^\prime$-proximal pair. This follows by Claim 3. 
 $\hfill\Box \Box$
 \end{proof}

\begin{theorem}
\label{main1} Let $(X,T)$ be minimal weak mixing, and $A\ne\emptyset$ a finite metric space. Let $S$ be a skew-product map of $X\times A$. Then, for every minimal set $M\subseteq X \times A$ of $S$, $(M, S|_{M})$ is $LYS_\varepsilon$, for some $\varepsilon >0$. Moreover, $S$ is $LYC_\varepsilon$ and for any $a\in A$, there is a dense $\varepsilon$-scrambled set $D_a\subset X\times \{ a\}$ of type $G_\delta$.
\end{theorem}

\begin{proof} 
The first part  follows by Lemmas \ref{lyp}, \ref{min}, \ref{trans} and Theorem \ref{lys}, the last statement by Lemma \ref{uncount}. To finish the argument, fix an $a\in A$, and let  $Y:=X\times\{ a\}$. 
Denote by  $R$ the set of $\varepsilon$-Li-Yorke pairs $(x,y)$ in $Y\times Y$. By Lemma \ref{uncount} there is an uncountable  dense scrambled set $D_a\subset Y$ such that every distinct points in $D_a$ form an $\varepsilon$-Li-Yorke pair.
\cbd
\end{proof}

\section{Infinite type skew-product extensions of weak mixing systems}

Theorem \ref{lys} and hence, Theorem \ref{main1} can be generalized to certain types of skew-product maps of $X\times A$, where $(X,T)$ is minimal weak mixing, and $A$ is infinite compact. As a sample we provide the following. Recall that an {\it adding machine} or {\it odometer} related to a sequence $p_1,p_2,\cdots$ of primes is a system $(X,\tau )$, where $X=\prod _{j\ge 1} X_j$, $X_j=\{0, 1, \cdots ,p_j-1\}$, and $\tau (x_1x_2x_3\cdots )=x_1x_2x_3\cdots + 1000\cdots$, when adding is modulo $p_j$ at the $j$th position from the left to the right, see, e.g., \cite{BK}. Obviously $X$ is a Cantor-type set.
 
\begin{theorem}
\label{main2} Let $(X,T)$ be minimal weak mixing, $Y$ a Cantor-type set, and $S: X\times Y\to X\times Y$ a skew-product map, $S(x,y)=(T(x), R_x(y))$ such that, for every $x\in X$, $R_x$ is an odometer, or the identity. Then $(X\times Y, S)$ is LYS.
\end{theorem}
 
\begin{proof}
It is similar to the proof of Theorem \ref{main1}. It is based on the following Lemma \ref{apr}, and on Theorem \ref{lys} and Lemma \ref{trans}.
\cbd
\end{proof}

\begin{lemma}
\label{apr} Let $X, Y$ and $S$ be as in Theorem \ref{main2}. Then for every $\delta >0$ there is an $m>0$ such that, for every $x\in X$, $y\in Y$ and $k\in\mathbb N$, $|y- R_x^{km}(y)|<\delta$.
\end{lemma}
 
\begin{proof}
 It suffices to show that, for every $\delta >0$ and for every $x\in X$, there is a decomposition of $Y$ into clopen  portions $Y_1, Y_2, \cdots , Y_m$ forming an $R_x$-periodic orbit, such that the diameter of every $Y_j$ is less than $\delta$. Assume the contrary. Then there is an increasing sequence $m_1<m_2<\cdots$ of positive integers, a 
sequence $x_1, x_2, \cdots$ in $X$, and a sequence $Y_1, Y_2, \cdots$ of clopen portions of $Y$ such that, for every $j$, ${\rm diam} (Y_j)\ge \delta$ and $Y_j$ is a periodic portion with respect to $R_{x_j}$ of period $m_j$. Taking a subsequence if necessary, we may assume that $\lim _{j\to\infty}x_j=x_0$, and $\lim _{j\to\infty} Y_j=Y_0$. Then $Y_0$ is a compact portion of $Y$ with \lq\lq infinite\rq\rq period, i.e., $R_{x_0}^n(Y_0)$ is disjoint from $Y_0$, for every $n>0$, contrary to the assumption that $R_{x_0}$ is an odometer, or the identity.
\cbd
\end{proof}

\section{Finite-type extensions and skew product systems}
   
  Here we show that the assumption in Theorem \ref{main1}, that the corresponding map is a skew-product map on $Y\times\{ 1,2,\cdots , n\}$ is  not too restrictive since, for certain but not all types of minimal weak mixing systems every $n$ to one  extension is a skew-product map, see Theorem \ref{main3}.
 On the other hand, by the next lemma and the subsequent remark, not every finite type extension of a minimal weak mixing system is (conjugate to) a skew-product map. Recall  (see, e.g.,  \cite{Kur} for details) that a {\it continuum} is a nonempty connected compact metric space. A continuum $X$ is {\it unicoherent} if for every two continua $A,B$ with $A\cup B=X$ the set $A\cap B$ is connected.
 
\begin{lemma}  
\label{ex} Let $(Y_1,\rho _1)$ be a  continuum which is not unicoherent. (In particular, let $Y_1$ be the circle). Let $Y=Y_0\times Y_1$, where $(Y_0,\rho _0 )$ is a compact metric space. Finally, let $S$ be a continuous map $Y\to Y$. Then for every $k\in\mathbb N$ there is an isometric extension $(X,T)$ of $(Y,S)$, with  factor map $\pi : X\to Y$ such that $\#\pi ^{-1}(y)=k$ for every $y\in Y$, and $(X,T)$ is not conjugate to a skew-product map $Y\times \{ 1,\cdots ,k\}\to Y\times \{ 1,\cdots ,k\}$.
\end{lemma}

\begin{proof}
Since $Y_1$ is not unicoherent there are continua $A, B\ne Y_1$ such that $A\cap B$ is not connected. Hence, $A\cap B= C\cup D$ where $C$, $D$ are nonempty disjoint compact sets. Assume first that $A\cap B=\{ a,b\}$. Let $k>1$ be an integer, $K_0\subset
  \mathbb S$ the set of points on the unit circle representing the $k$th roots of $1$ and, for any $t\in I$, let $\varphi _t$ be the rotation of the set $\mathbb S$ at angle $2t\pi /k$ in the positive direction. Thus, $\varphi _1(K_0)=K_0$. For $v\in B$ let $t(v)=\rho _1(a,v)/\rho _1(a,b)$. Let $X_1\subset Y_1\times\mathbb S$ be given by $X_1=(A\times K_0)\cup \bigcup _{v\in B\setminus A} (\{ v\}\times \varphi _{t(v)}(K_0))$, and let $X=Y_0\times X_1$; obviously, $X_1$ is connected.  For a $y\in Y$ denote by $y^\prime$ the projection of $y$ onto $Y_1$ and let $T: X\to X$ be such that, for $(y,z)\in X$ with $y\in Y$, 
  
\begin{equation}
\label{defg2}
 T(y,z)= \left\{
\begin{array}{ll}
(S(y), z)  & {\rm if} \ \ y^\prime, S(y)^\prime\in A,\\
(S(y), \varphi _{t(S(y)^\prime) - t(a)}(z))& {\rm if}  \ \  y^\prime\in A , S(y)^\prime \in B\setminus A,\\
 (S(y), \varphi _{t(b)-t(y^\prime )}(z))&  {\rm if} \ \ y^\prime \in B\setminus A, S(y)^\prime\in A, \\
 (S(y),\varphi _{t(S(y)^\prime )- t(y^\prime))}(z))  & {\rm if} \ \  y^\prime , S(y)^\prime \in B\setminus A.
\end{array} \right.
\end{equation}
  
Then $T$ is a continuous bijection $X\to X$. To finish the argument assume  $(X,T)$ is  conjugate to $Y\times\{ 1,\cdots ,k\}$. But then $X_1$ should be homeomorphic to $Y_1\times\{ 1,\cdots , k\}$, which is impossible since the first space is connected while the second one has  $k$ disjoint connected components.

In the general case when $C$ or $D$ is not a singleton, the argument is similar, we only take $t(v)={\rm dist}(C,v)/{\rm dist}(C,D)$ for $v\in Y_1\setminus A$.
\cbd
\end{proof}
  
\begin{remark} In \cite{F} there is an example of a  minimal weak mixing $(Y,S)$, where $Y$ is the 5th-dimensional torus. By the previous lemma, even in this case, there is a minimal $k$-extension $(X,T)$ of $(Y,S)$ which is not conjugate to a skew-product system.  
\end{remark}

\begin{lemma}
\label{decom} Let $(X,\rho)$ be a compact metric space such that every connected component of $X$ is nowhere dense in $X$. Then, for every $\delta >0$ there is a finite decomposition $X_1\cup X_2\cup\cdots\cup X_m$ of $X$ into disjoint compact subsets such that, for every $j$, $X_j$ is a subset of the $\delta$-neighborhood of a connected component of $X$.
 \end{lemma} 
 
\begin{proof} 
Let $X_\delta$ be the union of connected components of $X$ with diameter $\ge\delta$. Then $X_\delta$ is a closed set. Indeed, let $x_n\in X_\delta$ be such that $\lim _{n\to\infty}x_n=x$. Then there are connected components $K_n\subseteq X_\delta$ such that $x_n\in K_n$, for every $n$. Since the set of nonempty compact subsets of $X$, with the Hausdorff metric $\rho _H$, is a compact set, we may assume that there is a compact set $K\subset X$ such that $\lim _{n\to\infty}\rho _H(K_n,K)=0$.
Since $x\in K$, it suffices to show that $K$ is a connected component of $X$ with $K\subset X_\delta$. Obviously, ${\rm diam} (K) \ge\delta$. To show that $K$ is connected, assume the contrary. Then there are disjoint closed sets $G,H$  with $K=G\cup H$ such that $K\cap G\ne\emptyset\ne H\cap K$. Let  $G^\prime$, $H^\prime$ be disjoint closed neighborhoods of $G$ and $H$, respectively. Then, for every sufficiently large $n$, $K_n\subseteq G^\prime\cup H^\prime$, $K_n\cap G^\prime\ne\emptyset\ne K_n\cap H^\prime$ which is a contradiction.  Thus, $X_\delta$ is a closed set.

Next we show that for every  connected component $K\subset X_\delta$ there is a compact neighborhood $U(K)=U$ of $K$ such that $X\setminus U$ is  compact, and $U$ is contained in the open $\delta$-neighborhood $V$ of  $K$. Since $K$ is a component, for every $x\in X\setminus V$ there is a decomposition of $X$ into disjoint compact sets $G_x, H_x$ such that $G_x$ is a neighborhood of $x$ and $K\subseteq H_x$. Since $X\setminus V$ is compact, there is a finite cover $G_{x_1}\cup G_{x_2}\cup\cdots\cup G_{x_k}=X\setminus V$. Take $U (K)=H_{x_1}\cap\cdots\cap H_{x_k}$.

To finish the proof it suffices to take a finite cover $W=W_1\cup W_2\cup\cdots\cup W_s$ of $X_\delta$ consisting of disjoint compact sets such that every $W_j$ is a set $U(K_j)$ with $K_j\subseteq X_\delta$ a connected component. Then $X\setminus W$ is a compact set which can be divided into finitely disjoint compact sets with diam $\le\delta$.
\cbd
\end{proof}

 We say that a set $A$ in a metric space is {\it $\delta$-separated} if $\rho (u,v)\ge \delta$ for every distinct $u,v\in A$.

\begin{lemma}
\label{sep} Let $X,Y$ be compact metric spaces, $n>0$ an integer, and $\pi :X\to Y$ a continuous map such that, for every $y\in Y$, $\#\pi ^{-1}(y)=n$. Assume that $(Y,S)$ is a minimal system. Then there is a $\delta _0>0$ such that every set $\pi ^{-1}(y)$ is $\delta _0$-separated.
 \end{lemma}
 
\begin{proof} 
For $\delta >0$ let $Y_\delta$ be the set of $y\in Y$ such that $\pi ^{-1}(y)$ is $\delta$-separated. Then $Y_\delta$ is a compact set. Indeed, let $y_j\in Y_\delta$ such that $\lim _{j\to\infty} y_j=y_0$. Since the space of nonempty compact subsets  of $X$, equipped with the Hausdorff metric $\rho _H$, is  a compact space there is a subsequence $j_1<j_2<\cdots$ and a set $A\subset X$ such that $\lim _{k\to\infty}\rho _H(\pi ^{-1}(y_{j_k}), A)=0$. By the continuity, $\pi (A)=y_0$  and $\# A=n$. Hence $A=\pi ^{-1}(y_0)$ is $\delta$-separated, i.e., $y_0\in Y_\delta$. Since every $\pi^{-1}(y)$ is finite, $\bigcup _{j>0} Y_{1/j}=Y$.
By the Baire category theorem there is a $k>0$ such that $Y_{1/k}$ has nonempty interior. Since $Y$ is minimal, there is an $m>0$ such that $\bigcup _{0\le j\le m}S^{-j}(Y_{1/k})=Y$. By the continuity of $S$ there is a $\delta _0>0$ such that $S^{-j}(Y_{1/k})\subset Y_{\delta _0}$, $0\le j\le m$. Consequently, $Y=Y_{\delta _0}$. 
\cbd
\end{proof}

\begin{lemma}
\label{con}Let $(Y,S)$ be a factor of $(X,T)$, with factor map $\pi :X\to Y$. Assume that $(Y,S)$ is minimal, weak mixing, not connected, and  such that every subcontinuum of $Y$ is unicoherent. Finally, let $n>0$ be an integer such that, for every $y\in Y$, $\#\pi ^{-1}(y)=n$. Then $(X,T)$ is conjugate to a skew-product map $F: Y\times N\to Y\times N$, where $N=\{ 1,2,\cdots , n\}$. 
\end{lemma}

\begin{proof} 
By Lemma \ref{sep} there is a $\delta _0>0$ such that $\pi ^{-1}(y)$ is $\delta _0$-separated, for every $y\in Y$. Let $0<\eta <\delta _0/3$ be such that, for every $u,v\in X$, $\rho (u,v)<2\eta$ implies $\rho (T(u),T(v))<\delta _0/3$.  
 Since $(Y, S)$ is weak mixing and not connected every connected component of $(Y,\rho )$ is  nowhere dense. By Lemma \ref{decom}, there is a finite decomposition $Y_1\cup Y_2\cup\cdots \cup Y_m$ of $Y$ into disjoint compact sets such that  $Y_j$ is contained in the $\eta$-neighborhood of a connected component $P_k$ of $Y$. Let $p_k\in P_k$ and let $U_0$ be a compact $\eta$-neighborhood of $p_k$.
 Define a continuous map $\psi _k:\pi^{-1}(U_0\cap P_k)\to (U_0\cap P_k)\times N$. If $Y_k\subset U_0$ extend $\psi _k$ continuously onto $\pi ^{-1}(Y_k)$; by the choice of $\eta$ this extension is uniquely determined by $\psi _k$ restricted to $\pi ^{-1}(U_0\cap P_k)$. Otherwise take $U_1$ the compact $\eta$-neighborhood of $U_0$ and extend $\psi _k$ continuously to a map $\pi ^{-1}(U_1\cap P_k)\to (U_1\cap P_k)\times N$, etc. Since $P_k$ is compact  after finite number of steps $\psi _k$ is continuously extended onto $\pi ^{-1}(P_k)$ such that $\psi _k(\pi^{-1}(P_k))=P_k\times N$. Since $ P_k$ is unicoherent continuum, this extension is uniquely determined by $\psi _k$ on $\pi ^{-1}(U_0\cap P_k)$. Finally, by the choice of $\eta$, $\psi _k$ can be continuously (and uniquely) extended onto $\piÊ^{-1}(Y_k)$.  To finish the argument take $\psi =\psi _1\cup\cdots \cup \psi _m$ which is a continuous bijective map $X\to Y\times N$, and take $F=\psi\circ T\circ\psi ^{-1}$.
\cbd
\end{proof}

\begin{theorem}
\label{main3} Let $(Y,S)$ be  minimal, weak mixing, not connected, and  such that every subcontinuum of $Y$ is unicoherent. Let $n>0$ be an integer, and let $(X,T)$ be an extension of $(Y,S)$ such that $\#\pi ^{-1}(y)=n$ for every $y\in Y$. Finally, let $M\subseteq X$ be a minimal set.  Then $(M,T|_M)$ is LYS. 
\end{theorem}

\begin{proof} It follows by Lemma \ref{con} and Theorem \ref{main1}.
\cbd
\end{proof}

\section{Li-Yorke sensitivity and Li-Yorke chaos}

In \cite{AK} there is a problem whether {\it LYS} implies {\it LYC}; the converse implication obviously is not true. Here we show that under some additional conditions, the answer is positive. This significantly restricts the class of systems $(X,T)$ for which the implication need not hold. To simplify the argument, we will use the following notation. Given a system $(X,T)$ denote by $Dist$ the set of distal pairs $(x,y)\in X\times X$, and by $Asym _\varepsilon$ the set of $\varepsilon$-asymptotic pairs $(x,y)$ in $X\times X$.\\

\begin{theorem}
\label{main4} Let $(X, T)$ be $LYS$. Assume there is a non-empty open set $H\subset X$ such that $(H\times H) \cap Dist $ has empty interior or (equivalently) that $(H\times H)\cap Dist $ is a set of the first Baire category. Then there is an $\varepsilon >0$ such that $(X,T)$ is $LYC_\varepsilon$.
\end{theorem}

\begin{proof}   
By Lemma \ref{asym} there is an $\varepsilon >0$ such that $T$ is $LYS_\varepsilon$, and $Asym _\varepsilon$ is a first category set. It is easy to see that $Dist $ and $Asym _{\varepsilon}$ are  $F_\sigma$ sets hence, by the Baire category theorem,  $(H\times H)\cap Dist $ is of the first category if and only if it has the empty interior. Assume $(H\times H)\cap Dist$ is a first category set and put
$L=X\times X\setminus (Dist \cup Asym _\varepsilon )$. Then $L$ is a $G_\delta$ set dense in $H\times H$. By Lemmas \ref{FT} and \ref{uncount}, there is an uncountable set $D\subset H$ such that $D\times D\setminus\Delta\subset L$. Obviously, $L$ is the set of $\varepsilon$-Li-Yorke pairs in $X\times X$. Hence, $D$ is an $\varepsilon$-scrambled set for $(X,T)$ and hence, $T$ is $LYC_\varepsilon$.
\cbd
\end{proof}

\begin{remark} 
For a minimal  $(X,T)$ which is both $LYS$ and $LYC$, the set $Dist $ can be very large. In \cite{M1} there is an example of such a system, even without a weak mixing factor such that the set $Dist $ contains an open dense subset of $X\times X$. 
\end{remark}

\section*{Acknowledgements}

The author would like to express her thanks to Zuzana Roth and Samuel Roth for their essential conception. They proved Theorem \ref{lys} for a set $A$  containing exactly two points, which contributed to prove of Theorem~\ref{lys}. The author also sincerely thanks Professor Jaroslav Sm{\' i}tal for his very kind suggestions and excellent support.


\bibliographystyle{elsarticle-num}
\bibliography{<your-bib-database>}



\end{document}